\documentclass[11pt,a4paper]{article}
\usepackage[pdftex]{graphicx}
\usepackage{amsmath,amssymb,amsfonts,amsthm}
\numberwithin{equation}{section}
\usepackage{indentfirst}
\usepackage{enumitem} 
\usepackage[margin=1.3in]{geometry}
\usepackage{fancyhdr}

\usepackage[colorlinks=true,linkcolor=magenta,citecolor=blue,urlcolor=cyan]{hyperref}
\usepackage{titlesec}
\usepackage{etoolbox}
\usepackage{graphicx}  
\usepackage{changepage}
\theoremstyle{plain}
\newtheorem{theorem}{Theorem}[section]

\newtheorem{remark}{Remark}[section]

\makeatletter
\renewcommand{\maketitle}{
	\begin{center}
		{\Large\bfseries{\@title}\par}
		\vskip 1em
		{\normalsize
			\lineskip .5em
			\begin{tabular}[t]{c}
				\@author
			\end{tabular}\par}
		\vskip 1.5em
	\end{center}
}
\makeatother
\renewenvironment{abstract}{
	\begin{adjustwidth}{1.3cm}{1.3cm}
		\noindent{\large\bfseries{A{\scriptsize BSTRACT.}}}
	}{
	\end{adjustwidth}
}

\usepackage{color}
\usepackage{xcolor}
\usepackage[normalem]{ulem} 
\usepackage{soul}

\usepackage{graphicx} 
\usepackage{xstring}  

\usepackage{xstring}
\usepackage{graphicx} 

\usepackage{marvosym}
\begin{document}
	
	\title{On Ramanujan's $q$-Continued Fractions of Order Thirty-Four and Sixty-Eight}
    \vspace{0.9cm}

	\author{ \textbf {Dipika Sarkar and S. N. Fathima}}
	
	\maketitle
	
	\begin{abstract}
        We derived $q$-continued fractions $X_i(q)$ of order thirty-four and continued fractions $Y_i(q)$ of order sixty-eight from a general continued fraction identity of Ramanujan, where $i=1,2,3,4,5,6,7$ and $8$. We established some theta-function identities, and one has been proved for the continued fractions $X_i(q)$ and $Y_i(q)$. Furthermore, we obtained results on vanishing coefficients arising from these continued fractions and their reciprocals. As an application of the theta-function identities for $Y_i(q)$, we derived certain color partition identities.
	    \vspace{0.5cm}
		
        \vspace{0.5cm}
		
	\end{abstract}
	
	\bigskip
	
	\vspace{0.5em}
	\section{Introduction}
	
	For any complex numbers $\lambda$ and $q$, define the $q$-product $(\lambda; q)_\infty$ as
	\begin{equation}
		(\lambda; q)_\infty := \prod_{n=0}^{\infty} \left(1 - \lambda q^n \right), \quad |q| < 1.
		\tag{1.1}
	\end{equation}
	
	For brevity, we often write
	\[
	(\lambda_1; q)_\infty (\lambda_2; q)_\infty (\lambda_3; q)_\infty \cdots (\lambda_l; q)_\infty
	= (\lambda_1, \lambda_2, \lambda_3, \cdots, \lambda_l; q)_\infty.
	\]
	
	The Ramanujan’s general theta-function $f(a,b)$ \cite[p.~34]{Berndt} is defined as
	\begin{equation}
		f(a,b) = \sum_{n=-\infty}^{\infty} a^{\frac{n(n+1)}{2}} b^{\frac{n(n-1)}{2}}, 
		\quad |ab| < 1.
		\tag{1.2}
	\end{equation}
	
	In terms of $f(a,b)$, Jacobi’s triple product identity \cite[p.~35, Entry 19]{Berndt}
	can be stated as
	\begin{equation}
		f(a,b) = (-a;ab)_\infty (-b;ab)_\infty (ab;ab)_\infty
		= (-a, -b, ab; ab)_\infty.
		\tag{1.3}
	\end{equation}
	
	Three useful special cases of $f(a,b)$ are the theta-functions $\phi(q)$, $\psi(q)$ and $f(-q)$ 
	\cite[p.~36, Entry 22(i)--(iii)]{Berndt} given by
	
	\begin{equation}
		\phi(q) := f(q,q) = \sum_{n=-\infty}^{\infty} q^{n^2}
		= \frac{(-q; -q)_\infty}{(q; -q)_\infty},
		\tag{1.4}
	\end{equation}
	
	\begin{equation}
		\psi(q) := f(q,q^3) = \sum_{n=0}^{\infty} q^{\frac{n(n+1)}{2}}
		= \frac{(q^2; q^2)_\infty}{(q; q^2)_\infty},
		\tag{1.5}
	\end{equation}
	
	\begin{equation}
		f(-q) := f(-q,-q^2) = \sum_{n=-\infty}^{\infty} (-1)^n q^{\frac{n(3n-1)}{2}}
		= (q;q)_\infty.
		\tag{1.6}
	\end{equation}
	
	Ramanujan also defined the function $\chi(q)$ \cite[p.~36, Entry 22(iv)]{Berndt} as
	\begin{equation}
		\chi(q) = (-q; q^2)_\infty.
		\tag{1.7}
	\end{equation}
	
	One of Ramanujan’s remarkable contributions is in the field of $q$-continued fractions. 
	Ramanujan recorded many continued fractions in his notebooks and the most famous among them 
	is the Rogers--Ramanujan continued fraction $R(q)$ is defined by
	
	\begin{equation}
		R(q) := q^{1/5}\frac{(q,q^{4};q^{5})_{\infty}}{(q^{2},q^{3};q^{5})_{\infty}}
		= q^{1/5}\frac{f(-q,-q^{4})}{f(-q^{2},-q^{3})}
		= \cfrac{q^{1/5}}{1+\cfrac{q}{1+\cfrac{q^{2}}{1+\cfrac{q^{3}}{1+\cdots}}}},
		\qquad |q|<1.
		\tag{1.8}
	\end{equation}
	
	The Rogers--Ramanujan continued fraction $R(q)$ is often referred to as the 
	continued fraction of order five. Ramanujan also offered some theta-function 
	identities and modular relations for the continued fraction $R(q)$. A description
	of these can be found in \cite{Berndt}. Ramanujan also recorded some general 
	continued fraction identities in his notebook. 
	
	For example, Ramanujan recorded the following general continued fraction 
	identity \cite[p.~24, Entry~12]{Berndt}.
	
	Suppose that $a$, $b$, and $q$ are complex numbers with $|ab|<1$ and $|q|<1$, 
	or that $a=b^{2l+1}$ for some integer $l$. Then
	
	\begin{equation}
		\frac{(a^{2}q^{3};q^{4})_{\infty}(b^{2}q^{3};q^{4})_{\infty}}
		{(a^{2}q;q^{4})_{\infty}(b^{2}q;q^{4})_{\infty}}
		=
		\cfrac{1}{1-ab+
			\cfrac{(a-bq)(b-aq)}
			{(1-ab)(q^{2}+1)+
				\cfrac{(a-bq^{3})(b-aq^{3})}
				{(1-ab)(q^{4}+1)+\cdots}}
		}.
		\tag{1.9}
	\end{equation}
	
	By specialising the values of $a$ and $b$, and taking suitable powers of $q$, 
	one can obtain $q$-continued fractions of particular order which satisfy 
	theta-function identities analogous to those of $R(q)$.
    
    In this paper, we deal with the $q$-continued fractions of order thirty-four and sixty-eight .By replacing $q$ by $q^{17/2}$ in ${(1.9)}$, setting 
	$\{a=q^{1/4},\, b=q^{33/4}\}$, $\{a=q^{3/4},\, b=q^{31/4)}\}$, $\{a=q^{5/4},\, b=q^{29/4)}\}$, $\{a=q^{7/4},\, b=q^{27/4)}\}$, $\{a=q^{9/4},\, b=q^{25/4)}\}$, $\{a=q^{11/4},\, b=q^{23/4)}\}$, $\{a=q^{13/4},\, b=q^{21/4)}\}$ and $\{a=q^{15/4},\, b=q^{19/4)}\}$, simplifying using the results
$\{(q^{42};q^{34})_{\infty} = (q^{8};q^{34})_{\infty}/(1-q^{8})\}$, 
$\{(q^{41};q^{34})_{\infty} = (q^{7};q^{34})_{\infty}/(1-q^{7})\}$, 
$\{(q^{40};q^{34})_{\infty} = (q^{6};q^{34})_{\infty}/(1-q^{6})\}$, 
$\{(q^{39};q^{34})_{\infty} = (q^{5};q^{34})_{\infty}/(1-q^{5})\}$, 
$\{(q^{38};q^{34})_{\infty} = (q^{4};q^{34})_{\infty}/(1-q^{4})\}$, 
$\{(q^{37};q^{34})_{\infty} = (q^{3};q^{34})_{\infty}/(1-q^{3})\}$, 
$\{(q^{36};q^{34})_{\infty} = (q^{2};q^{34})_{\infty}/(1-q^{2})\}$
and
$\{(q^{35};q^{34})_{\infty} = (q;q^{34})_{\infty}/(1-q)\}$.
We obtain the following eight continued fractions of order thirty-four, respectively.:
	\begin{align}
		X_1(q) :&= q^{1/4}\frac{(q^{8},q^{26};q^{34})_{\infty}}{(q^{9},q^{25};q^{34})_{\infty}}
		= q^{1/4}\frac{f(-q^{8},-q^{26})}{f(-q^{9},-q^{25})}  \notag\\
		&=
		\cfrac{q^{1/4}(1-q^{8})}
		{(1-q^{17/2})
			+\cfrac{q^{17/2}(1-q^{1/2})(1-q^{33/2})}
			{(1-q^{17/2})(1+q^{17})
				+\cfrac{q^{17/2}(1-q^{35/2})(1-q^{67/2})}
				{(1-q^{17/2})(1+q^{34})+\cdots}}}.
		\tag{1.10}
	\end{align}
	
	\begin{align}
		X_2(q) :&= q^{3/4}\frac{(q^{7},q^{27};q^{34})_{\infty}}{(q^{10},q^{24};q^{34})_{\infty}}
		= q^{3/4}\frac{f(-q^{7},-q^{27})}{f(-q^{10},-q^{24})}  \notag\\
		&=
		\cfrac{q^{3/4}(1-q^{7})}
		{(1-q^{17/2})
			+\cfrac{q^{17/2}(1-q^{3/2})(1-q^{31/2})}
			{(1-q^{17/2})(1+q^{17})
				+\cfrac{q^{17/2}(1-q^{37/2})(1-q^{65/2})}
				{(1-q^{17/2})(1+q^{34})+\cdots}}}.
		\tag{1.11}
	\end{align}
	
	\begin{align}
		X_3(q) :&= q^{5/4}\frac{(q^{6},q^{28};q^{34})_{\infty}}{(q^{11},q^{23};q^{34})_{\infty}}
		= q^{5/4}\frac{f(-q^{6},-q^{28})}{f(-q^{11},-q^{23})}  \notag\\
		&=
		\cfrac{q^{5/4}(1-q^{6})}
		{(1-q^{17/2})
			+\cfrac{q^{17/2}(1-q^{5/2})(1-q^{29/2})}
			{(1-q^{17/2})(1+q^{17})
				+\cfrac{q^{17/2}(1-q^{39/2})(1-q^{63/2})}
				{(1-q^{17/2})(1+q^{34})+\cdots}}}.
		\tag{1.12}
	\end{align}
	
	\begin{align}
		X_4(q) :&= q^{7/4}\frac{(q^{5},q^{29};q^{34})_{\infty}}{(q^{12},q^{22};q^{34})_{\infty}}
		= q^{7/4}\frac{f(-q^{5},-q^{29})}{f(-q^{12},-q^{22})}  \notag\\
		&=
		\cfrac{q^{7/4}(1-q^{5})}
		{(1-q^{17/2})
			+\cfrac{q^{17/2}(1-q^{7/2})(1-q^{27/2})}
			{(1-q^{17/2})(1+q^{17})
				+\cfrac{q^{17/2}(1-q^{41/2})(1-q^{61/2})}
				{(1-q^{17/2})(1+q^{34})+\cdots}}}.
		\tag{1.13}
	\end{align}
	
	\begin{align}
		X_5(q) :&= q^{9/4}\frac{(q^{4},q^{30};q^{34})_{\infty}}{(q^{13},q^{21};q^{34})_{\infty}}
		= q^{9/4}\frac{f(-q^{4},-q^{30})}{f(-q^{13},-q^{21})}  \notag\\
		&=
		\cfrac{q^{9/4}(1-q^{4})}
		{(1-q^{17/2})
			+\cfrac{q^{17/2}(1-q^{9/2})(1-q^{25/2})}
			{(1-q^{17/2})(1+q^{17})
				+\cfrac{q^{17/2}(1-q^{43/2})(1-q^{59/2})}
				{(1-q^{17/2})(1+q^{34})+\cdots}}}.
		\tag{1.14}
	\end{align}
	
	\begin{align}
		X_6(q) :&= q^{11/4}\frac{(q^{3},q^{31};q^{34})_{\infty}}{(q^{14},q^{20};q^{34})_{\infty}}
		= q^{11/4}\frac{f(-q^{3},-q^{31})}{f(-q^{14},-q^{20})}  \notag\\
		&=
		\cfrac{q^{11/4}(1-q^{3})}
		{(1-q^{17/2})
			+\cfrac{q^{17/2}(1-q^{11/2})(1-q^{23/2})}
			{(1-q^{17/2})(1+q^{17})
				+\cfrac{q^{17/2}(1-q^{45/2})(1-q^{57/2})}
				{(1-q^{17/2})(1+q^{34})+\cdots}}}.
		\tag{1.15}
	\end{align}
	\begin{align}
		X_7(q) :&= q^{13/4}\frac{(q^{2},q^{32};q^{34})_{\infty}}{(q^{15},q^{19};q^{34})_{\infty}}
		= q^{13/4}\frac{f(-q^{2},-q^{32})}{f(-q^{15},-q^{19})}  \notag\\
		&=
		\cfrac{q^{13/4}(1-q^{2})}
		{(1-q^{17/2})
			+\cfrac{q^{17/2}(1-q^{13/2})(1-q^{21/2})}
			{(1-q^{17/2})(1+q^{17})
				+\cfrac{q^{17/2}(1-q^{47/2})(1-q^{55/2})}
				{(1-q^{17/2})(1+q^{34})+\cdots}}}.
		\tag{1.16}
	\end{align}
	
	\begin{align}
		X_8(q) :&= q^{15/4}\frac{(q,q^{33};q^{34})_{\infty}}{(q^{16},q^{18};q^{34})_{\infty}}
		= q^{15/4}\frac{f(-q,-q^{33})}{f(-q^{16},-q^{18})}  \notag\\
		&=
		\cfrac{q^{15/4}(1-q)}
		{(1-q^{17/2})
			+\cfrac{q^{17/2}(1-q^{15/2})(1-q^{19/2})}
			{(1-q^{17/2})(1+q^{17})
				+\cfrac{q^{17/2}(1-q^{49/2})(1-q^{53/2})}
				{(1-q^{17/2})(1+q^{34})+\cdots}}}.
		\tag{1.17}
	\end{align}
	
	Similarly, to obtain the $q$-continued fraction of order sixty-eight, replacing $q$ by $q^{17}$ in ${(1.9)}$, setting 
	$\{a=q^{1},\, b=q^{16}\}$, $\{a=q^{2},\, b=q^{15}\}$, $\{a=q^{3},\, b=q^{14)}\}$, $\{a=q^{4},\, b=q^{13)}\}$, $\{a=q^{5},\, b=q^{12)}\}$, $\{a=q^{6},\, b=q^{11)}\}$, $\{a=q^{7},\, b=q^{10)}\}$ and $\{a=q^{8},\, b=q^{9)}\}$, simplifying using the results
$\{(q^{83};q^{68})_{\infty} = (q^{15};q^{68})_{\infty}/(1-q^{15})\}$, 
$\{(q^{81};q^{68})_{\infty} = (q^{13};q^{68})_{\infty}/(1-q^{13})\}$,
$\{(q^{79};q^{68})_{\infty} = (q^{11};q^{68})_{\infty}/(1-q^{11})\}$, 
$\{(q^{77};q^{68})_{\infty} = (q^{9};q^{68})_{\infty}/(1-q^{9})\}$,
$\{(q^{75};q^{68})_{\infty} = (q^{7};q^{68})_{\infty}/(1-q^{7})\}$, 
$\{(q^{73};q^{68})_{\infty} = (q^{5};q^{68})_{\infty}/(1-q^{5})\}$,
$\{(q^{71};q^{68})_{\infty} = (q^{3};q^{68})_{\infty}/(1-q^{3})\}$ and
$\{(q^{69};q^{68})_{\infty} = (q;q^{68})_{\infty}/(1-q)\}$.
We obtain the following eight continued fractions of order sixty-eight from the equation ${(1.9)}$, which are given by,
	\begin{align}
		Y_1(q) :&= q\frac{(q^{15},q^{53};q^{68})_{\infty}}{(q^{19},q^{49};q^{68})_{\infty}}
		= q\frac{f(-q^{15},-q^{53})}{f(-q^{19},-q^{49})}  \notag\\
		&=
		\cfrac{q(1-q^{15})}
		{(1-q^{17})
			+\cfrac{q^{17}(1-q^{32})(1-q^{2})}
			{(1-q^{17})(1+q^{34})
				+\cfrac{q^{17}(1-q^{36})(1-q^{66})}
				{(1-q^{17})(1+q^{68})+\cdots}}}.
		\tag{1.18}
	\end{align}
	\begin{align}
		Y_2(q) :&= q^{2}\frac{(q^{13},q^{55};q^{68})_{\infty}}{(q^{21},q^{47};q^{68})_{\infty}}
		= q^{2}\frac{f(-q^{13},-q^{55})}{f(-q^{21},-q^{47})}  \notag\\
		&=
		\cfrac{q^{2}(1-q^{13})}
		{(1-q^{17})
			+\cfrac{q^{17}(1-q^{30})(1-q^{4})}
			{(1-q^{17})(1+q^{34})
				+\cfrac{q^{17}(1-q^{38})(1-q^{64})}
				{(1-q^{17})(1+q^{68})+\cdots}}}.
		\tag{1.19}
	\end{align}
	\begin{align}
		Y_3(q) :&= q^{3}\frac{(q^{11},q^{57};q^{68})_{\infty}}{(q^{23},q^{45};q^{68})_{\infty}}
		= q^{3}\frac{f(-q^{11},-q^{57})}{f(-q^{23},-q^{45})}  \notag\\
		&=
		\cfrac{q^{3}(1-q^{11})}
		{(1-q^{17})
			+\cfrac{q^{17}(1-q^{28})(1-q^{6})}
			{(1-q^{17})(1+q^{34})
				+\cfrac{q^{17}(1-q^{40})(1-q^{62})}
				{(1-q^{17})(1+q^{68})+\cdots}}}.
		\tag{1.20}
	\end{align}
	\begin{align}
		Y_4(q) :&= q^{4}\frac{(q^{9},q^{59};q^{68})_{\infty}}{(q^{25},q^{43};q^{68})_{\infty}}
		= q^{4}\frac{f(-q^{9},-q^{59})}{f(-q^{25},-q^{43})}  \notag\\
		&=
		\cfrac{q^{4}(1-q^{9})}
		{(1-q^{17})
			+\cfrac{q^{17}(1-q^{26})(1-q^{8})}
			{(1-q^{17})(1+q^{34})
				+\cfrac{q^{17}(1-q^{42})(1-q^{60})}
				{(1-q^{17})(1+q^{68})+\cdots}}}.
		\tag{1.21}
	\end{align}
	\begin{align}
		Y_5(q) :&= q^{5}\frac{(q^{7},q^{61};q^{68})_{\infty}}{(q^{27},q^{41};q^{68})_{\infty}}
		= q^{5}\frac{f(-q^{7},-q^{61})}{f(-q^{27},-q^{41})}  \notag\\
		&=
		\cfrac{q^{5}(1-q^{7})}
		{(1-q^{17})
			+\cfrac{q^{17}(1-q^{24})(1-q^{10})}
			{(1-q^{17})(1+q^{34})
				+\cfrac{q^{17}(1-q^{44})(1-q^{58})}
				{(1-q^{17})(1+q^{68})+\cdots}}}.
		\tag{1.22}
	\end{align}
	\begin{align}
		Y_6(q) :&= q^{6}\frac{(q^{5},q^{63};q^{68})_{\infty}}{(q^{29},q^{39};q^{68})_{\infty}}
		= q^{6}\frac{f(-q^{5},-q^{63})}{f(-q^{29},-q^{39})}  \notag\\
		&=
		\cfrac{q^{6}(1-q^{5})}
		{(1-q^{17})
			+\cfrac{q^{17}(1-q^{22})(1-q^{12})}
			{(1-q^{17})(1+q^{34})
				+\cfrac{q^{17}(1-q^{46})(1-q^{56})}
				{(1-q^{17})(1+q^{68})+\cdots}}}.
		\tag{1.23}
	\end{align}
	\begin{align}
		Y_7(q) :&= q^{7}\frac{(q^{3},q^{65};q^{68})_{\infty}}{(q^{31},q^{37};q^{68})_{\infty}}
		= q^{7}\frac{f(-q^{3},-q^{65})}{f(-q^{31},-q^{37})}  \notag\\
		&=
		\cfrac{q^{7}(1-q^{3})}
		{(1-q^{17})
			+\cfrac{q^{17}(1-q^{20})(1-q^{14})}
			{(1-q^{17})(1+q^{34})
				+\cfrac{q^{17}(1-q^{48})(1-q^{54})}
				{(1-q^{17})(1+q^{68})+\cdots}}}.
		\tag{1.24}
	\end{align}
	\begin{align}
		Y_8(q) :&= q^{8}\frac{(q,q^{67};q^{68})_{\infty}}{(q^{33},q^{35};q^{68})_{\infty}}
		= q^{8}\frac{f(-q,-q^{67})}{f(-q^{33},-q^{35})}  \notag\\
		&=
		\cfrac{q^{8}(1-q)}
		{(1-q^{17})
			+\cfrac{q^{17}(1-q^{18})(1-q^{16})}
			{(1-q^{17})(1+q^{34})
				+\cfrac{q^{17}(1-q^{50})(1-q^{52})}
				{(1-q^{17})(1+q^{68})+\cdots}}}.
		\tag{1.25}
	\end{align}

	In Section 2, we prove some theta-function identities for the continued fractions $X_i(q)$ and $Y_i(q)$.
	In Section 3, we derived some vanishing coefficient results arising from the continued fractions $X_i(q)$ and $Y_i(q)$, where $i=1,2,3,4,5,6,7$ and $8.$
    In Section 4, we have shown that color partition identities can be obtained from the theta-function identities by using continued fractions with suitable examples.
	
	\section{Theta-function identities}
	
	In this section, we establish theta-function and modular identities for the continued fractions $X_i(q)$ and $Y_i(q)$.
	
	\begin{theorem}
		We have
		\[
		\text{(a)}\quad 
		\frac{1}{X_1(q)} + X_1(q)
		= \frac{\phi(-q^{17/2})\, f(q^{1/2}, q^{33/2})}
		{q^{1/4}\psi(q^{17})\, f(-q^{8},-q^{9})}.
		\]
		\[
		\text{(b)}\quad 
		\frac{1}{X_1(q)} - X_1(q)
		= \frac{\phi(q^{17/2})\, f(-q^{1/2}, -q^{33/2})}
		{q^{1/4}\psi(q^{17})\, f(-q^{8},-q^{9})}.
		\]
		\[
		\text{(c)}\quad 
		\frac{1}{X_2(q)} + X_2(q)
		= \frac{\phi(-q^{17/2})\, f(q^{3/2}, q^{31/2})}
		{q^{3/4}\psi(q^{17})\, f(-q^{7},-q^{10})}.
		\]
		\[
		\text{(d)}\quad 
		\frac{1}{X_2(q)} - X_2(q)
		= \frac{\phi(q^{17/2})\, f(-q^{3/2},-q^{31/2})}
		{q^{3/4}\psi(q^{17})\, f(-q^{7},-q^{10})}.
		\]
		\[
		\text{(e)}\quad 
		\frac{1}{X_3(q)} + X_3(q)
		= \frac{\phi(-q^{17/2})\, f(q^{5/2}, q^{29/2})}
		{q^{5/4}\psi(q^{17})\, f(-q^{6},-q^{11})}.
		\]
		\[
		\text{(f)}\quad 
		\frac{1}{X_3(q)} - X_3(q)
		= \frac{\phi(q^{17/2})\, f(-q^{5/2},-q^{29/2})}
		{q^{5/4}\psi(q^{17})\, f(-q^{6},-q^{11})}.
		\]
		\[
		\text{(g)}\quad 
		\frac{1}{X_4(q)} + X_4(q)
		= \frac{\phi(-q^{17/2})\, f(q^{7/2}, q^{27/2})}
		{q^{7/4}\psi(q^{17})\, f(-q^{5},-q^{12})}.
		\]
		\[
		\text{(h)}\quad 
		\frac{1}{X_4(q)} - X_4(q)
		= \frac{\phi(q^{17/2})\, f(-q^{7/2},-q^{27/2})}
		{q^{7/4}\psi(q^{17})\, f(-q^{5},-q^{12})}.
		\]
		\[
		\text{(i)}\quad 
		\frac{1}{X_5(q)} + X_5(q)
		= \frac{\phi(-q^{17/2})\, f(q^{9/2}, q^{25/2})}
		{q^{9/4}\psi(q^{17})\, f(-q^{4},-q^{13})}.
		\]
		\[
		\text{(j)}\quad 
		\frac{1}{X_5(q)} - X_5(q)
		= \frac{\phi(q^{17/2})\, f(-q^{9/2},-q^{25/2})}
		{q^{9/4}\psi(q^{17})\, f(-q^{4},-q^{13})}.
		\]
		\[
		\text{(k)}\quad 
		\frac{1}{X_6(q)} + X_6(q)
		= \frac{\phi(-q^{17/2})\, f(q^{11/2}, q^{23/2})}
		{q^{11/4}\psi(q^{17})\, f(-q^{3},-q^{14})}.
		\]
		\[
		\text{(l)}\quad 
		\frac{1}{X_6(q)} - X_6(q)
		= \frac{\phi(q^{17/2})\, f(-q^{11/2},-q^{23/2})}
		{q^{11/4}\psi(q^{17})\, f(-q^{3},-q^{14})}.
		\]
		\[
		\text{(m)}\quad 
		\frac{1}{X_7(q)} + X_7(q)
		= \frac{\phi(-q^{17/2})\, f(q^{13/2}, q^{21/2})}
		{q^{13/4}\psi(q^{17})\, f(-q^{2},-q^{15})}.
		\]
		\[
		\text{(n)}\quad 
		\frac{1}{X_7(q)} - X_7(q)
		= \frac{\phi(q^{17/2})\, f(-q^{13/2},-q^{21/2})}
		{q^{13/4}\psi(q^{17})\, f(-q^{2},-q^{15})}.
		\]
		\[
		\text{(o)}\quad 
		\frac{1}{X_8(q)} + X_8(q)
		= \frac{\phi(-q^{17/2})\, f(q^{15/2}, q^{19/2})}
		{q^{15/4}\psi(q^{17})\, f(-q,-q^{16})}.
		\]
		\[
		\text{(p)}\quad 
		\frac{1}{X_8(q)} - X_8(q)
		= \frac{\phi(q^{17/2})\, f(-q^{15/2},-q^{19/2})}
		{q^{15/4}\psi(q^{17})\, f(-q,-q^{16})}.
		\]
		
	\end{theorem}
	
	\begin{proof}
		Here, we consider $X_1(q)$ only. One can prove the remaining  $X_2(q)$, $X_3(q)$, $X_4(q)$, $X_5(q)$, $X_6(q)$, $X_7(q)$ and $X_8(q)$ similarly.
		Using (1.10), we see that
		\begin{equation}
			\frac{1}{\sqrt{X_1(q)}} - \sqrt{X_1(q)}
			=
			\frac{f(-q^{9},-q^{25})-q^{1/4}f(-q^{8},-q^{26})}
			{\sqrt{{q^{1/4}}\, f(-q^{8},-q^{26})f(-q^{9},-q^{25})}}.
			\tag{2.1}
		\end{equation}
		
		From \cite[p.~46]{Berndt}, Entry 30 (ii) and (iii), we have
		\begin{equation}
			f(a,b)=f(a^3b,ab^3)+af(b/a,a^5b^3).
			\tag{2.2}
		\end{equation}
		
		Taking $\{a=-q^{1/4},\, b=q^{33/4}\}$ and $\{a=q^{1/4},\, b=-q^{33/4}\}$ in (1.14), we obtain
		\begin{equation}
			f(-q^{1/4},q^{33/4})
			=
			f(-q^{9},-q^{25})-q^{1/4}f(-q^{8},-q^{26})
			\tag{2.3}
		\end{equation}

		and
		\begin{equation}
			f(q^{1/4},-q^{33/4})
			=
			f(-q^{9},-q^{25})+q^{1/4}f(-q^{8},-q^{26}),
			\tag{2.4}
		\end{equation}
		respectively.
		
		Applying (2.3) in (2.1), we find that
		\begin{equation}
			\frac{1}{\sqrt{X_1(q)}}-\sqrt{X_1(q)}
			=
			\frac{f(-q^{1/4},q^{33/4})}
			{\sqrt{{q^{1/4}}\, f(-q^{8},-q^{26})f(-q^{9},-q^{25})}}.
			\tag{2.5}
		\end{equation}
		
		Similarly, from (1.10) and applying (2.4), we deduce that
		\begin{equation}
			\frac{1}{\sqrt{X_1(q)}}+\sqrt{X_1(q)}
			=
			\frac{f(q^{1/4},-q^{33/4})}
			{\sqrt{{q^{1/4}}\, f(-q^{8},-q^{26})f(-q^{9},-q^{25})}}.
			\tag{2.6}
		\end{equation}
		
		Combining (2.5) and (2.6), we obtain
		\begin{equation}
			\frac{1}{X_1(q)}-X_1(q)
			=
			\frac{f(-q^{1/4},q^{33/4})f(q^{1/4},-q^{33/4})}
			{q^{1/4}f(-q^{8},-q^{26})f(-q^{9},-q^{25})}.
			\tag{2.7}
		\end{equation}
		
		Again, from \cite[p.~46]{Berndt}, Entry 30 (i) and (iv), we have
		\begin{equation}
			f(a,ab^2)f(b,a^2b)=f(a,b)\psi(ab)
			\tag{2.8}
		\end{equation}
		
		and
		\begin{equation}
			f(a,b)f(-a,-b)=f(-a^2,-b^2)\phi(-ab).
			\tag{2.9}
		\end{equation}
		
		Setting $\{a=-q^{8},b=-q^{9}\}$ in (2.
		8) and $\{a=-q^{1/4},b=q^{33/4}\}$ in (2.9), we obtain
		\begin{equation}
			f(-q^{8},-q^{26})f(-q^{9},-q^{25})
			=
			f(-q^{8},-q^{9})\psi(q^{17}).
			\tag{2.10}
		\end{equation}
		
		and 
		\begin{equation}
			f(-q^{1/4},q^{33/4})f(q^{1/4},-q^{33/4})=f(-q^{1/2},-q^{33/2})\phi(q^{17/2}).
			\tag{2.11}
		\end{equation}
		respectively. Employing (2.10) and (2.11) in (2.7), we complete the proof of (a).
		Squaring (2.6) we obtain
		\begin{equation}
			\frac{1}{X_1(q)}+X_1(q)
			=
			\frac{f^2(q^{1/4},-q^{33/4})}
			{q^{1/4}f(-q^{8},-q^{26})f(-q^{9},-q^{25})}-2.
			\tag{2.12}
		\end{equation}
		From \cite[p.~46]{Berndt}, Entry 30 (v) and (vi), we have
		\begin{equation}
			f^2(a,b)=f(a^2,b^2)\phi(ab)+2af(b/a,a^3b)\psi(a^2b^2).
			\tag{2.13}
		\end{equation}
		setting $\{a=q^{1/4},b=-q^{33/4}\}$, we obtain
		\begin{equation}
			f^2(q^{1/4},-q^{33/4})=f(q^{1/2},q^{33/2})\phi(-q^{17/2})+2q^{1/4}f(-q^{8},-q^{9})\psi(q^{17}).
			\tag{2.14}
		\end{equation}
		Employing (2.14) and (2.10) in (2.12) and simplifying, we arrive at (b).
		Proofs of (c)-(p) are identical to the proofs of (a) and (b), so we omit them.
	\end{proof} 
	The proof of Theorem $2.2$ is similar to the proof of Theorem $2.1$, so we simply state the theorem and omit the proof.

	\begin{theorem}
		We have
		\[
		\text{(a)}\quad 
		\frac{1}{Y_1(q)} + Y_1(q)
		= \frac{\phi(-q^{17})\, f(q^{2}, q^{32})}
		{q\psi(q^{34})\, f(-q^{15},-q^{19})}.
		\]
		\[
		\text{(b)}\quad 
		\frac{1}{Y_1(q)} - Y_1(q)
		= \frac{\phi(q^{17})\, f(-q^{2}, -q^{32})}
		{q\psi(q^{34})\, f(-q^{15},-q^{19})}.
		\]
		\[
		\text{(c)}\quad 
		\frac{1}{Y_2(q)} + Y_2(q)
		= \frac{\phi(-q^{17})\, f(q^{4}, q^{30})}
		{q^{2}\psi(q^{34})\, f(-q^{13},-q^{21})}.
		\]
		\[
		\text{(d)}\quad 
		\frac{1}{Y_2(q)} - Y_2(q)
		= \frac{\phi(q^{17})\, f(-q^{4}, -q^{30})}
		{q^{2}\psi(q^{34})\, f(-q^{13},-q^{21})}.
		\]
			\[
		\text{(e)}\quad 
		\frac{1}{Y_3(q)} + Y_3(q)
		= \frac{\phi(-q^{17})\, f(q^{6}, q^{28})}
		{q^{3}\psi(q^{34})\, f(-q^{11},-q^{23})}.
		\]
		\[
		\text{(f)}\quad 
		\frac{1}{Y_3(q)} - Y_3(q)
		= \frac{\phi(q^{17})\, f(-q^{6}, -q^{28})}
		{q^{3}\psi(q^{34})\, f(-q^{11},-q^{23})}.
		\]
		\[
		\text{(g)}\quad 
		\frac{1}{Y_4(q)} + Y_4(q)
		= \frac{\phi(-q^{17})\, f(q^{8}, q^{26})}
		{q^{4}\psi(q^{34})\, f(-q^{9},-q^{25})}.
		\]
		\[
		\text{(h)}\quad 
		\frac{1}{Y_4(q)} - Y_4(q)
		= \frac{\phi(q^{17})\, f(-q^{8}, -q^{26})}
		{q^{4}\psi(q^{34})\, f(-q^{9},-q^{25})}.
		\]
		\[
		\text{(i)}\quad 
		\frac{1}{Y_5(q)} + Y_5(q)
		= \frac{\phi(-q^{17})\, f(q^{10}, q^{24})}
		{q^{5}\psi(q^{34})\, f(-q^{7},-q^{27})}.
		\]
		\[
		\text{(j)}\quad 
		\frac{1}{Y_5(q)} - Y_5(q)
		= \frac{\phi(q^{17})\, f(-q^{10}, -q^{24})}
		{q^{5}\psi(q^{34})\, f(-q^{7},-q^{27})}.
		\]
		\[
		\text{(k)}\quad 
		\frac{1}{Y_6(q)} + Y_6(q)
		= \frac{\phi(-q^{17})\, f(q^{12}, q^{22})}
		{q^{6}\psi(q^{34})\, f(-q^{5},-q^{29})}.
		\]
		\[
		\text{(l)}\quad 
		\frac{1}{Y_6(q)} - Y_6(q)
		= \frac{\phi(q^{17})\, f(-q^{12}, -q^{22})}
		{q^{6}\psi(q^{34})\, f(-q^{5},-q^{29})}.
		\]
		\[
		\text{(m)}\quad 
		\frac{1}{Y_7(q)} + Y_7(q)
		= \frac{\phi(-q^{17})\, f(q^{14}, q^{20})}
		{q^{7}\psi(q^{34})\, f(-q^{3},-q^{31})}.
		\]
		\[
		\text{(n)}\quad 
		\frac{1}{Y_7(q)} - Y_7(q)
		= \frac{\phi(q^{17})\, f(-q^{14}, -q^{20})}
		{q^{7}\psi(q^{34})\, f(-q^{3},-q^{31})}.
		\]
		\[
		\text{(o)}\quad 
		\frac{1}{Y_8(q)} + Y_8(q)
		= \frac{\phi(-q^{17})\, f(q^{16}, q^{18})}
		{q^{8}\psi(q^{34})\, f(-q,-q^{33})}.
		\]
		\[
		\text{(p)}\quad 
		\frac{1}{Y_8(q)} - Y_8(q)
		= \frac{\phi(q^{17})\, f(-q^{16}, -q^{18})}
		{q^{8}\psi(q^{34})\, f(-q,-q^{33})}.
		\]
	\end{theorem}

	\begin{theorem}
		For any non-negative integer $n$, we have for $i=1, 2, 3, 4, 5, 6, 7$ and $8$:
		\[
		\text{(a)}\quad
		X_i^{n}(q)\,X_i^{n}(-q)
		=
		(-1)^{\frac{n}{4}}\,X_i^{n}(q^2),
		\quad \text{if } n \equiv 0 \pmod{4},
		\]
		\[
		\text{where }\;
		(-1)^{\frac{n}{4}}
		=
		\begin{cases}
			 1, & \text{if } n \equiv 0 \pmod{8},\\[4pt]
			-1, & \text{if } n \equiv 4 \pmod{8}.
		\end{cases}
		\]
	\end{theorem}
	
	\begin{proof}
		Here, we consider the only case $i=1$. One can prove the remaining cases similarly.
		Using (1.10), we see that
		\begin{equation}
			X_1^{n}(q)X_1^{n}(-q)
			=
			(-1)^{n/4} q^{n/2}
			\frac{f^n(-q^{8},-q^{26})}{f^n(-q^{9},-q^{25})}
			\cdot
			\frac{f^n(q^{8},q^{26})}{f^n(q^{9},q^{25})}.
			\tag{2.15}
		\end{equation}
		
		Setting $\{a=q^{9},\, b=q^{25}\}$ and $\{a=q^{8},\, b=q^{26}\}$ in (2.9), we find that
		\begin{equation}
			f(q^{9},q^{25})f(-q^{9},-q^{25})=f(-q^{18},-q^{50})\phi(-q^{34}),
			\tag{2.16}
		\end{equation}
		and
		\begin{equation}
			f(q^{8},q^{26})f(-q^{8},-q^{26})=f(-q^{16},-q^{52})\phi(-q^{34}),
			\tag{2.17}
		\end{equation}
		respectively.
		
		Employing (2.16) and (2.17) in (2.15), we obtain
		\begin{equation}
			X_1^{n}(q)X_1^{n}(-q)
			=
			(-1)^{n/4} q^{n/2}
			\frac{f^n(-q^{16},-q^{52})}{f^n(-q^{18},-q^{50})}
			=
			(-1)^{n/4} X_1^{n}(q^2).
			\tag{2.18}
		\end{equation}
		
		Noting that $n/4$ is even when $n \equiv 0 \pmod{8}$ and odd when $n \equiv 4 \pmod{8}$ in (2.18), the proof of part (a) for $i=1$ is complete. The cases $i=2,3,4,5,6,7$ and $8$ follow similarly, and hence their proofs are omitted.
	\end{proof}
	
	The proof of Theorem $2.4$ is similar to the proof of Theorem $2.3$, so we simply state the theorem and omit the proof.

	\begin{theorem}
		For any non-negative integer $n$, we have for $i=1,2,3,4,5,6,7$ and $8$:
		\[
		\text{(a)}\quad
		Y_i^{n}(q)Y_i^{n}(-q)=
		\begin{cases}
			Y_i^{n}(q^2), & \text{if } n \equiv 0 \pmod{2},\\
			-\,Y_i^{n}(q^2), & \text{if } n \equiv 1 \pmod{2}.
		\end{cases}
		\]
	\end{theorem}

	\section{Vanishing coefficient in the series expansion}
	\begin{theorem}
		If
		\[
			X_1^*(q) = q^{-1/4} X_1(q)
			= \frac{(q^8, q^{26}; q^{34})_\infty}{(q^9, q^{25}; q^{34})_\infty}
			= \sum_{n=0}^{\infty} \xi_n q^n,
			\]
			Then we have
			\[
			\xi_{17n+6} = 0.
			\]
	\end{theorem}
	\begin{proof}
			
			Andrews and Bressoud \cite{Andrews} stated the following $p$-dissection formula:
			\[
			\frac{(q^x, q^x, q^{y+z}, q^{x-y-z}; q^x)_\infty}
			{(q^z, q^{x-z}, q^y, q^{x-y}; q^x)_\infty}
			=
			\sum_{j=0}^{w-1}
			q^{jy}
			\frac{(q^{wx}, q^{wx}, q^{wy+z+jx}, q^{(w-j)x-wy-z}; q^{wx})_\infty}
			{(q^{jx+z}, q^{(w-j)x-z}, q^{wy}, q^{(x-y)w}; q^{wx})_\infty}.
			\tag{3.1}
			\]
		where all of the powers of $q$ in each of the infinite products on the right-hand side must be multiples of $w$, and the integer $y$ must satisfy $\gcd(y, w) = 1$.
		\[
		\text{Substituting } x = 34,\; y = 9,\; z = 17,\; \text{and } w = 17 \text{ into (4.1), we obtain}
		\]
		\[
		\begin{aligned}
			\frac{(q^{34}, q^{34}, q^{26}, q^{8}; q^{34})_{\infty}}
			{(q^{17}, q^{17}, q^{9}, q^{25}; q^{34})_{\infty}}
			&= \frac{(q^{578}, q^{578}, q^{170}, q^{408}; q^{578})_{\infty}}
			{(q^{17}, q^{561}, q^{153}, q^{425}; q^{578})_{\infty}}
			+ q^{9}\frac{(q^{578}, q^{578}, q^{204}, q^{374}; q^{578})_{\infty}}
			{(q^{51}, q^{527}, q^{153}, q^{425}; q^{578})_{\infty}} \\
			&\quad + q^{18}\frac{(q^{578}, q^{578}, q^{238}, q^{340}; q^{578})_{\infty}}
			{(q^{85}, q^{493}, q^{153}, q^{425}; q^{578})_{\infty}}
			+ q^{27}\frac{(q^{578}, q^{578}, q^{272}, q^{306}; q^{578})_{\infty}}
			{(q^{119}, q^{459}, q^{153}, q^{425}; q^{578})_{\infty}} \\
			&\quad + q^{36}\frac{(q^{578}, q^{578}, q^{306}, q^{272}; q^{578})_{\infty}}
			{(q^{153}, q^{425}, q^{153}, q^{425}; q^{578})_{\infty}}
			+ q^{45}\frac{(q^{578}, q^{578}, q^{340}, q^{238}; q^{578})_{\infty}}
			{(q^{187}, q^{391}, q^{153}, q^{425}; q^{578})_{\infty}} \\
			&\quad + q^{54}\frac{(q^{578}, q^{578}, q^{374}, q^{204}; q^{578})_{\infty}}
			{(q^{221}, q^{357}, q^{153}, q^{425}; q^{578})_{\infty}}
			+ q^{63}\frac{(q^{578}, q^{578}, q^{408}, q^{170}; q^{578})_{\infty}}
			{(q^{255}, q^{323}, q^{153}, q^{425}; q^{578})_{\infty}} \\
			&\quad + q^{72}\frac{(q^{578}, q^{578}, q^{442}, q^{136}; q^{578})_{\infty}}
			{(q^{289}, q^{289}, q^{153}, q^{425}; q^{578})_{\infty}}
			+ q^{81}\frac{(q^{578}, q^{578}, q^{476}, q^{102}; q^{578})_{\infty}}
			{(q^{323}, q^{255}, q^{153}, q^{425}; q^{578})_{\infty}} \\
			&\quad + q^{90}\frac{(q^{578}, q^{578}, q^{510}, q^{68}; q^{578})_{\infty}}
			{(q^{357}, q^{221}, q^{153}, q^{425}; q^{578})_{\infty}}
			+ q^{99}\frac{(q^{578}, q^{578}, q^{544}, q^{34}; q^{578})_{\infty}}
			{(q^{391}, q^{187}, q^{153}, q^{425}; q^{578})_{\infty}} \\
			&\quad + q^{108}\frac{(q^{578}, q^{578}, q^{578}, q^{0}; q^{578})_{\infty}}
			{(q^{425}, q^{153}, q^{153}, q^{425}; q^{578})_{\infty}}
			+ q^{117}\frac{(q^{578}, q^{578}, q^{612}, q^{-34}; q^{578})_{\infty}}
			{(q^{459}, q^{119}, q^{153}, q^{425}; q^{578})_{\infty}} \\
			&\quad + q^{126}\frac{(q^{578}, q^{578}, q^{646}, q^{-68}; q^{578})_{\infty}}
			{(q^{493}, q^{85}, q^{153}, q^{425}; q^{578})_{\infty}}
			+ q^{135}\frac{(q^{578}, q^{578}, q^{680}, q^{-102}; q^{578})_{\infty}}
			{(q^{527}, q^{51}, q^{153}, q^{425}; q^{578})_{\infty}}\\
			&\quad +  q^{144}\frac{(q^{578}, q^{578}, q^{714}, q^{-136}; q^{578})_{\infty}}
			{(q^{561}, q^{17}, q^{153}, q^{425}; q^{578})_{\infty}}
		\end{aligned}
		\tag{3.2}
		\]

	\[
	\text{Multiplying both sides of  equation(3.2) by }
	{(q^{17}; q^{34})_{\infty}^{2}}/{(q^{34}; q^{34})_{\infty}^{2}}
	\text{ and then simplifying, we obtain}
	\]
	
	\[
	\sum_{n=0}^{\infty} \xi_n q^n
	=
	\left(
	\begin{aligned}
		&\frac{(q^{17}, q^{153}, q^{425}, q^{561}; q^{578})_{\infty}}
		{(q^{170}, q^{408}; q^{578})_{\infty}} \\
		&\qquad \times
		\frac{
			(q^{51}, q^{85}, q^{119}, q^{187}, q^{221}, q^{255}, q^{289}, q^{323}, q^{357}, q^{391}, q^{459}, q^{493}, q^{527}; q^{578})_{\infty}^{2}
		}{
			(q^{34}, q^{68}, q^{102}, q^{136}, q^{204}, q^{238}, q^{272}, q^{306}, q^{340}, q^{374}, q^{442}, q^{476}, q^{510}, q^{544}; q^{578})_{\infty}^{2}
		}
	\end{aligned}
	\right)
	\]
	
	\[
	\quad + q^{9}\left(
	\begin{aligned}
		&\frac{(q^{51}, q^{153}, q^{425}, q^{527}; q^{578})_{\infty}}
		{(q^{204}, q^{374}; q^{578})_{\infty}} \\
		&\qquad \times
		\frac{
			(q^{17}, q^{85}, q^{119}, q^{187}, q^{221}, q^{255}, q^{289}, q^{323}, q^{357}, q^{391}, q^{459}, q^{493}, q^{561}; q^{578})_{\infty}^{2}
		}{
			(q^{34}, q^{68}, q^{102}, q^{136}, q^{170}, q^{238}, q^{272}, q^{306}, q^{340}, q^{408}, q^{442}, q^{476}, q^{510}, q^{544}; q^{578})_{\infty}^{2}
		}
	\end{aligned}
	\right)
	\]
	
	\[
	\quad +q^{18}\left(
	\begin{aligned}
		&\frac{(q^{85}, q^{153}, q^{425}, q^{493}; q^{578})_{\infty}}
		{(q^{238}, q^{340}; q^{578})_{\infty}} \\
		&\qquad \times
		\frac{
			(q^{17}, q^{51}, q^{119}, q^{187}, q^{221}, q^{255}, q^{289}, q^{323}, q^{357}, q^{391}, q^{459}, q^{527}, q^{561}; q^{578})_{\infty}^{2}
		}{
		(q^{34}, q^{68}, q^{102}, q^{136}, q^{170},  q^{204}, q^{272}, q^{306}, q^{374}, q^{408}, q^{442}, q^{476}, q^{510}, q^{544}; q^{578})_{\infty}^{2}
	    }
	\end{aligned} 
	\right)
	\]
	
	\[
	\quad +q^{27}\left(
	\begin{aligned}
		&\frac{(q^{119}, q^{153}, q^{425}, q^{459}; q^{578})_{\infty}}
		{(q^{272}, q^{306}; q^{578})_{\infty}} \\ 
		&\qquad \times
		\frac{
			(q^{17}, q^{51}, q^{85}, q^{187}, q^{221}, q^{255}, q^{289}, q^{323}, q^{357}, q^{391},  q^{493}, q^{527},  q^{561}; q^{578})_{\infty}^{2}}{(q^{34}, q^{68}, q^{102}, q^{136}, q^{170},  q^{204}, q^{238}, q^{340}, q^{374}, q^{408},  q^{442}, q^{476}, q^{510}, q^{544}; q^{578})_{\infty}^{2}} 
	\end{aligned}
	\right)
	\]
	
	\[
	\quad +q^{36}\left(
	\begin{aligned}
			&\frac{(q^{153},q^{425}; q^{578})_{\infty}}
			{(q^{272}, q^{306}; q^{578})_{\infty}} \\
			&\qquad \times
			\frac{(q^{17}, q^{51}, q^{85}, q^{119}, q^{187}, q^{221}, q^{255}, q^{289}, q^{323}, q^{357}, q^{391}, q^{459}, q^{493}, q^{527}, q^{561}; q^{578})_{\infty}^{2}}{(q^{34}, q^{68}, q^{102}, q^{136}, q^{170}, q^{204}, q^{238}, q^{340}, q^{374}, q^{408}, q^{442}, q^{476}, q^{510}, q^{544}; q^{578})_{\infty}^{2}}
	\end{aligned}
	\right) 
	\]
	
	\[
	\quad +q^{45}\left(
	\begin{aligned}
		&\frac{(q^{153}, q^{187}, q^{391}, q^{425}; q^{578})_{\infty}}{(q^{238}, q^{340}; q^{578})_{\infty}} \\
		&\qquad \times
		\frac{(q^{17}, q^{51}, q^{85}, q^{119}, q^{221}, q^{255}, q^{289}, q^{323}, q^{357}, q^{459}, q^{493}, q^{527}, q^{561}; q^{578})_{\infty}^{2}}{(q^{34}, q^{68}, q^{102}, q^{136}, q^{170}, q^{204}, q^{272}, q^{306}, q^{374}, q^{408},  q^{442}, q^{476}, q^{510}, q^{544}; q^{578})_{\infty}^{2}}
	\end{aligned} 
	\right)
	\]
	
	\[
	\quad +q^{54}\left(
	\begin{aligned}
			&\frac{(q^{153}, q^{221},  q^{357},  q^{425}; q^{578})_{\infty}}{(q^{204}, q^{374}; q^{578})_{\infty}} \\
			&\qquad \times
			\frac{(q^{17}, q^{51}, q^{85}, q^{119}, q^{187}, q^{255}, q^{289}, q^{323}, q^{391}, q^{459}, q^{493}, q^{527}, q^{561}; q^{578})_{\infty}^{2}}{(q^{34}, q^{68}, q^{102}, q^{136}, q^{170}, q^{238}, q^{272}, q^{306}, q^{340}, q^{408}, q^{442}, q^{476}, q^{510}, q^{544}; q^{578})_{\infty}^{2}} 
	\end{aligned}
	\right)
	\]
	
	\[
	\quad +q^{63}\left(
	\begin{aligned}
	&\frac{
		( q^{153}, q^{255}, q^{323}, q^{425}; q^{578})_{\infty}}{(q^{170}, q^{408}; q^{578})_{\infty}} \\
		&\qquad \times
		\frac{(q^{17}, q^{51}, q^{85}, q^{119}, q^{187}, q^{221}, q^{289}, q^{357}, q^{391}, q^{459}, q^{493}, q^{527}, q^{561}; q^{578})_{\infty}^{2}}{(q^{34}, q^{68}, q^{102}, q^{136}, q^{204}, q^{238}, q^{272}, q^{306}, q^{340}, q^{374}, q^{442}, q^{476}, q^{510}, q^{544}; q^{578})_{\infty}^{2}}
	\end{aligned}
	\right)
	\]

	\[
	\quad +q^{72}\left(
	\begin{aligned}
		&\frac{(q^{153}, q^{289}, q^{425}; q^{578})_{\infty}}{(q^{136}, q^{442}; q^{578})_{\infty}} \\
		&\qquad \times
		\frac{(q^{17}, q^{51}, q^{85}, q^{119}, q^{187}, q^{221}, q^{255}, q^{323}, q^{357}, q^{391}, q^{459}, q^{493}, q^{527}, q^{561}; q^{578})_{\infty}^{2}}{(q^{34}, q^{68}, q^{102}, q^{170}, q^{204}, q^{238}, q^{272}, q^{306}, q^{340}, q^{374}, q^{408}, q^{476}, q^{510}, q^{544}; q^{578})_{\infty}^{2}}
	\end{aligned}
	\right)
	\]
	
	\[
	\quad +q^{81}\left(
	\begin{aligned}
		&\frac{(q^{153}, q^{255}, q^{323}, q^{425}; q^{578})_{\infty}}{(q^{102}, q^{476}; q^{578})_{\infty}} \\
		&\qquad \times
		\frac{(q^{17}, q^{51}, q^{85}, q^{119}, q^{187}, q^{221}, q^{289}, q^{357}, q^{391}, q^{459}, q^{493}, q^{527}, q^{561}; q^{578})_{\infty}^{2}}{(q^{34}, q^{68}, q^{136}, q^{170},  q^{204}, q^{238}, q^{272}, q^{306}, q^{340}, q^{374}, q^{408}, q^{442}, q^{510}, q^{544}; q^{578})_{\infty}^{2}} 
	\end{aligned}\right)
	\]
	
	\[
	\quad +q^{90}\left(
	\begin{aligned}
		&\frac{(q^{153}, q^{221},  q^{357}, q^{425}; q^{578})_{\infty}}{(q^{68}, q^{510}; q^{578})_{\infty}} \\
		&\qquad \times
		\frac{(q^{17}, q^{51}, q^{85}, q^{119}, q^{187}, q^{255}, q^{289}, q^{323}, q^{391}, q^{459}, q^{493}, q^{527}, q^{561}; q^{578})_{\infty}^{2}}{(q^{34}, q^{102}, q^{136}, q^{170}, q^{204}, q^{238}, q^{272}, q^{306}, q^{340}, q^{374}, q^{408}, q^{442}, q^{476}, q^{544}; q^{578})_{\infty}^{2}}
	\end{aligned} 
	\right)
	\]
	
	\[
	\quad +q^{99}\left(
	\begin{aligned}
		&\frac{(q^{153}, q^{187}, q^{425}, q^{391}; q^{578})_{\infty}}{(q^{34}, q^{544}; q^{578})_{\infty}} \\
		&\qquad \times
		\frac{(q^{17}, q^{51}, q^{85}, q^{119}, q^{221}, q^{255}, q^{289}, q^{323}, q^{357}, q^{459}, q^{493}, q^{527}, q^{561}; q^{578})_{\infty}^{2}}{(q^{68}, q^{102}, q^{136}, q^{170}, q^{204}, q^{238}, q^{272}, q^{306}, q^{340}, q^{374}, q^{408}, q^{442}, q^{476}, q^{510}; q^{578})_{\infty}^{2}}
	\end{aligned}
	\right) 
	\]
	
	\[
	\quad + q^{117}\left(
	\begin{aligned}
		&(q^{-34}, q^{119}, q^{153}, q^{425}, q^{459}; q^{578})_{\infty} \\
		&\qquad \times
		\frac{
			(q^{17}, q^{51}, q^{85}, q^{187}, q^{221}, q^{255}, q^{289}, q^{323}, q^{357}, q^{391}, q^{493}, q^{527}, q^{561}; q^{578})_{\infty}^{2}
		}{
			(q^{34}, q^{68}, q^{102}, q^{136}, q^{170}, q^{204}, q^{238}, q^{272}, q^{306}, q^{340}, q^{374}, q^{408}, q^{442}, q^{476}, q^{510}, q^{544}; q^{578})_{\infty}^{2}
		}
	\end{aligned}
	\right)
	\]
	
	\[
	\quad +q^{126}\left(
	\begin{aligned}
		&(q^{-68}, q^{85}, q^{153}, q^{425}, q^{493}; q^{578})_{\infty} \\
		&\qquad \times
		\frac{(q^{17}, q^{51}, q^{119}, q^{187}, q^{221}, q^{255}, q^{289}, q^{323}, q^{357}, q^{391}, q^{459}, q^{527}, q^{561}; q^{578})_{\infty}^{2}}{(q^{34}, q^{68}, q^{102}, q^{136},q^{170}, q^{204}, q^{238}, q^{272}, q^{306}, q^{340}, q^{374}, q^{408}, q^{442}, q^{476}, q^{510}, q^{544}; q^{578})_{\infty}^{2}}
    \end{aligned} 
    \right)
	\]
	
	\[
	\quad +q^{135}\left(
	\begin{aligned}
		&(q^{-102}, q^{51}, q^{153}, q^{425}, q^{527}; q^{578})_{\infty} \\
		&\qquad \times
		\frac{(q^{17}, q^{85}, q^{119}, q^{187}, q^{221}, q^{255}, q^{289}, q^{323}, q^{357}, q^{391}, q^{459}, q^{493}, q^{561}; q^{578})_{\infty}^{2}}{(q^{34}, q^{68}, q^{102}, q^{136},q^{170}, q^{204}, q^{238}, q^{272}, q^{306}, q^{340}, q^{374}, q^{408}, q^{442}, q^{476}, q^{510}, q^{544}; q^{578})_{\infty}^{2}}
	\end{aligned}
	\right)
	\]
	
	\[
	\quad +q^{144}\left(
	\begin{aligned}
		&(q^{-136}, q^{17}, q^{153}, q^{425}, q^{561}; q^{578})_{\infty} \\
		&\qquad \times
		\frac{(q^{51}, q^{85}, q^{119}, q^{187}, q^{221}, q^{255}, q^{289}, q^{323}, q^{357}, q^{391}, q^{459}, q^{493}, q^{527}; q^{578})_{\infty}^{2}}{(q^{34}, q^{68}, q^{102}, q^{136}, q^{170}, q^{204}, q^{238}, q^{272}, q^{306}, q^{340}, q^{374}, q^{408}, q^{442}, q^{476}, q^{510}, q^{544}; q^{578})_{\infty}^{2}} 
	\end{aligned}
	\right)
	\tag{3.3}
	\]

	where we have used the identity $f(-1,a)=0$ from \cite[pp.~34, Entry 18(iii)]{Berndt}. 
	Since the right-hand side of the equation ${(3.3)}$ contains no term of the form $q^{17n+6}$, it follows that the coefficient of $q^{17n+6}$ is zero, which yields the desired result.
	\end{proof}
	
	\begin{remark}
		The following table represents the remaining vanishing coefficients in the $q$-series expansions connected with the continued fractions ${X_i(q)}$, where $i=2,3,4,5,6,7$ and $8$.
	\end{remark}
	\[
	\renewcommand{\arraystretch}{2}
	\begin{array}{|l|c|}
		\hline
		\text{q-series/continued fractions} & \text{Vanishing coefficients} \\ 
		\hline
    	\frac{1}{X_2^*(q)} = q^{-3/4} X_2(q)
		= \frac{(q^{7}, q^{27}; q^{34})_\infty}{(q^{10}, q^{24}; q^{34})_\infty}
		= \sum_{n=0}^{\infty} \xi'_n q^n,
		& \xi'_{17n+6} = 0 \\ 
		\hline

		X_3^*(q) = q^{-5/4} X_3(q)
		= \frac{(q^6, q^{28}; q^{34})_\infty}{(q^{11}, q^{23}; q^{34})_\infty}
		= \sum_{n=0}^{\infty} \phi_n q^n,
		& \phi_{17n+2} = 0 \\ 
		\hline
		
		\frac{1}{X_4^*(q)} = q^{-7/4} X_4(q)
		= \frac{(q^{5}, q^{29}; q^{34})_\infty}{(q^{12}, q^{22}; q^{34})_\infty}
		= \sum_{n=0}^{\infty} \phi'_n q^n,
		& \phi'_{17n+2} = 0 \\ 
		\hline
		
		X_5^*(q) = q^{-9/4} X_5(q)
		= \frac{(q^4, q^{30}; q^{34})_\infty}{(q^{13}, q^{21}; q^{34})_\infty}
		= \sum_{n=0}^{\infty} \zeta_n q^n
		& \zeta_{17n+11} = 0 \\ 
		\hline
		
		\frac{1}{X_6^*(q)} = q^{-11/4} X_6(q)
		= \frac{(q^{3}, q^{31}; q^{34})_\infty}{(q^{14}, q^{20}; q^{34})_\infty}
		= \sum_{n=0}^{\infty} \zeta'_n q^n,
		& \zeta'_{17n+11} = 0 \\ 
		\hline
		
		X_7^*(q) = q^{-13/4} X_7(q)
		= \frac{(q^2, q^{32}; q^{34})_\infty}{(q^{15}, q^{19}; q^{34})_\infty}
		= \sum_{n=0}^{\infty} \eta_n q^n,
		& \eta_{17n+16} = 0 \\ 
		\hline
		
	    \frac{1}{X_8^*(q)} = q^{-15/4} X_8(q)
     	= \frac{(q^{1}, q^{33}; q^{34})_\infty}{(q^{16}, q^{18}; q^{34})_\infty}
    	= \sum_{n=0}^{\infty} \eta'_n q^n,
		& \eta'_{17n+16} = 0 \\ 
		\hline
        \hline
		
	\end{array}
	\]
	
	\begin{theorem}
		If
		\[
		\frac{1}{Y_7^*(q)} = q^{-7} Y_7(q)
		= \frac{(q^{3}, q^{65}; q^{68})_\infty}{(q^{31}, q^{37}; q^{68})_\infty}
		= \sum_{n=0}^{\infty} \xi_n q^n,
		\]
		then we have
		\[
		\xi_{34n+28} = 0.
		\]
	\end{theorem}
	\begin{proof}
			\[
		\text{Substituting } x = 64,\; y = 3,\; z = 34,\; \text{and } w = 34 \text{ into (3.1), we obtain}
		\]
		\[
		\begin{aligned}
			\frac{(q^{68}, q^{68}, q^{37}, q^{31}; q^{68})_{\infty}} 
			{(q^{34}, q^{34}, q^{3}, q^{65}; q^{68})_{\infty}} 
			&= \frac{(q^{2312}, q^{2312}, q^{136}, q^{2176}; q^{2312})_{\infty}}
			{(q^{34}, q^{2278}, q^{102}, q^{2210}; q^{2312})_{\infty}}
			+ q^{3}\frac{(q^{2312}, q^{2312}, q^{204}, q^{2108}; q^{2312})_{\infty}}
			{(q^{102}, q^{2210}, q^{102}, q^{2210}; q^{2312})_{\infty}} \\
			&\quad + q^{6}\frac{(q^{2312}, q^{2312}, q^{272}, q^{2040}; q^{2312})_{\infty}}
			{(q^{170}, q^{2142}, q^{102}, q^{2210}; q^{2312})_{\infty}}
			+ q^{9}\frac{(q^{2312}, q^{2312}, q^{340}, q^{1972}; q^{2312})_{\infty}}
			{(q^{238}, q^{2074}, q^{102}, q^{2210}; q^{2312})_{\infty}} \\
			&\quad + q^{12}\frac{(q^{2312}, q^{2312}, q^{408}, q^{1904}; q^{2312})_{\infty}}
			{(q^{306}, q^{2006}, q^{102}, q^{2210}; q^{2312})_{\infty}}
			+ q^{15}\frac{(q^{2312}, q^{2312}, q^{476}, q^{1836}; q^{2312})_{\infty}}
			{(q^{374}, q^{1938}, q^{102}, q^{2210}; q^{2312})_{\infty}} \\
			&\quad + q^{18}\frac{(q^{2312}, q^{2312}, q^{544}, q^{1768}; q^{2312})_{\infty}}
			{(q^{442}, q^{1870}, q^{102}, q^{2210}; q^{2312})_{\infty}} 
			+ q^{21}\frac{(q^{2312}, q^{2312}, q^{612}, q^{1700}; q^{2312})_{\infty}}
			{(q^{510}, q^{1802}, q^{102}, q^{2210}; q^{2312})_{\infty}} \\
			&\quad + q^{24}\frac{(q^{2312}, q^{2312}, q^{680}, q^{1632}; q^{2312})_{\infty}}
			{(q^{578}, q^{1734}, q^{102}, q^{2210}; q^{2312})_{\infty}}
			+ q^{27}\frac{(q^{2312}, q^{2312}, q^{748}, q^{1564}; q^{2312})_{\infty}}
			{(q^{646}, q^{1666}, q^{102}, q^{2210}; q^{2312})_{\infty}} \\
			&\quad + q^{30}\frac{(q^{2312}, q^{2312}, q^{816}, q^{1496}; q^{2312})_{\infty}}
			{(q^{714}, q^{1598}, q^{102}, q^{2210}; q^{2312})_{\infty}}
			+ q^{33}\frac{(q^{2312}, q^{2312}, q^{884}, q^{1428}; q^{2312})_{\infty}}
			{(q^{782}, q^{1530}, q^{102}, q^{2210}; q^{2312})_{\infty}} \\
			&\quad + q^{36}\frac{(q^{2312}, q^{2312}, q^{952}, q^{1360}; q^{2312})_{\infty}}
			{(q^{850}, q^{1462}, q^{102}, q^{2210}; q^{2312})_{\infty}}
			+ q^{39}\frac{(q^{2312}, q^{2312}, q^{1020}, q^{1292}; q^{2312})_{\infty}}
			{(q^{918}, q^{1394}, q^{102}, q^{2210}; q^{2312})_{\infty}} \\
			&\quad +q^{42}\frac{(q^{2312}, q^{2312}, q^{1088}, q^{1224}; q^{2312})_{\infty}}
			{(q^{986}, q^{1326}, q^{102}, q^{2210}; q^{2312})_{\infty}} 
			+ q^{45}\frac{(q^{2312}, q^{2312}, q^{1156}, q^{1156}; q^{2312})_{\infty}}
			{(q^{1054}, q^{1258}, q^{102}, q^{2210}; q^{2312})_{\infty}} \\
			&\quad +  q^{48}\frac{(q^{2312}, q^{2312}, q^{1224}, q^{1088}; q^{2312})_{\infty}}
			{(q^{1122}, q^{1190}, q^{102}, q^{2210}; q^{2312})_{\infty}}
			+ q^{51}\frac{(q^{2312}, q^{2312}, q^{1292}, q^{1020}; q^{2312})_{\infty}}
			{(q^{1190}, q^{1122}, q^{102}, q^{2210}; q^{2312})_{\infty}} \\
			&\quad + q^{54}\frac{(q^{2312}, q^{2312}, q^{1360}, q^{952}; q^{2312})_{\infty}}
			{(q^{1258}, q^{1054}, q^{102}, q^{2210}; q^{2312})_{\infty}} 
			+ q^{57}\frac{(q^{2312}, q^{2312}, q^{1428}, q^{884}; q^{2312})_{\infty}}
			{(q^{1326}, q^{986}, q^{102}, q^{2210}; q^{2312})_{\infty}} \\
			&\quad + q^{60}\frac{(q^{2312}, q^{2312}, q^{1496}, q^{816}; q^{2312})_{\infty}}
			{(q^{1394}, q^{918}, q^{102}, q^{2210}; q^{2312})_{\infty}}
			+ q^{63}\frac{(q^{2312}, q^{2312}, q^{1564}, q^{748}; q^{2312})_{\infty}}
			{(q^{1462}, q^{850}, q^{102}, q^{2210}; q^{2312})_{\infty}} \\
			&\quad +  q^{66}\frac{(q^{2312}, q^{2312}, q^{1632}, q^{680}; q^{2312})_{\infty}}
			{(q^{1530}, q^{782}, q^{102}, q^{2210}; q^{2312})_{\infty}}
			+ q^{69}\frac{(q^{2312}, q^{2312}, q^{1700}, q^{612}; q^{2312})_{\infty}}
			{(q^{1598}, q^{714}, q^{102}, q^{2210}; q^{2312})_{\infty}} \\
			&\quad + q^{72}\frac{(q^{2312}, q^{2312}, q^{1768}, q^{544}; q^{2312})_{\infty}}
			{(q^{1666}, q^{646}, q^{102}, q^{2210}; q^{2312})_{\infty}}
			+ q^{75}\frac{(q^{2312}, q^{2312}, q^{1836}, q^{476}; q^{2312})_{\infty}}
			{(q^{1734}, q^{578}, q^{102}, q^{2210}; q^{2312})_{\infty}} \\
			&\quad +  q^{78}\frac{(q^{2312}, q^{2312}, q^{1904}, q^{408}; q^{2312})_{\infty}}
			{(q^{1802}, q^{510}, q^{102}, q^{2210}; q^{2312})_{\infty}} 
			+ q^{81}\frac{(q^{2312}, q^{2312}, q^{1972}, q^{340}; q^{2312})_{\infty}}
			{(q^{1870}, q^{442}, q^{102}, q^{2210}; q^{2312})_{\infty}} \\
			&\quad + q^{84}\frac{(q^{2312}, q^{2312}, q^{2040}, q^{272}; q^{2312})_{\infty}}
			{(q^{1938}, q^{374}, q^{102}, q^{2210}; q^{2312})_{\infty}}
			+ q^{87}\frac{(q^{2312}, q^{2312}, q^{2108}, q^{204}; q^{2312})_{\infty}}
			{(q^{2006}, q^{306}, q^{102}, q^{2210}; q^{2312})_{\infty}} \\
			&\quad + q^{90}\frac{(q^{2312}, q^{2312}, q^{2176}, q^{136}; q^{2312})_{\infty}}
			{(q^{2074}, q^{238}, q^{102}, q^{2210}; q^{2312})_{\infty}}
			+ q^{93}\frac{(q^{2312}, q^{2312}, q^{2244}, q^{68}; q^{2312})_{\infty}}
			{(q^{2142}, q^{170}, q^{102}, q^{2210}; q^{2312})_{\infty}} \\
			&\quad +  q^{96}\frac{(q^{2312}, q^{2312}, q^{2312}, q^{0}; q^{2312})_{\infty}}
			{(q^{2210}, q^{102}, q^{102}, q^{2210}; q^{2312})_{\infty}}
			+ q^{99}\frac{(q^{2312}, q^{2312}, q^{2380}, q^{-68}; q^{2312})_{\infty}}
			{(q^{2278}, q^{34}, q^{102}, q^{2210}; q^{2312})_{\infty}} 
		\end{aligned}
		\tag{3.4}
		\]
		
		Multiplying ${(q^{34}; q^{68})_{\infty}^{2}}/{(q^{68}; q^{68})_{\infty}^{2}}$ on both sides of $(3.4)$ and using $f(-1, a)=0$, then extracting the terms involving $q^{34n+28}$, we arrive at the result.
	\end{proof}
	
	\begin{remark}
		The following table represents the remaining vanishing coefficients in the $q$-series expansions connected with the continued fractions ${Y_i(q)}$, where $i=1,2,3,4,5,6,7$ and $8$.
	\end{remark}
	\[
	\renewcommand{\arraystretch}{2}
	\begin{array}{|l|c|}
		\hline
		\text{q-series/continued fractions} & \text{Vanishing coefficients} \\ 
		\hline
		{Y_1^*(q)} = q^{-1} Y_1(q)
		= \frac{(q^{15}, q^{53}; q^{68})_\infty}{(q^{19}, q^{49}; q^{68})_\infty}
		= \sum_{n=0}^{\infty} \alpha_n q^n,
		& \alpha_{34n+14} = 0 \\ 
		\hline

		\frac{1}{{Y}_1^*(q)} = q^{-1} Y_1(q)
		= \frac{(q^{15}, q^{53}; q^{68})_\infty}{(q^{19}, q^{49}; q^{68})_\infty}
		= \sum_{n=0}^{\infty} \alpha'_n q^n,
		& \alpha'_{34n+16} = 0 \\ 
		\hline
		
		{Y_2^*(q)} = q^{-2} Y_2(q)
		= \frac{(q^{13}, q^{55}; q^{68})_\infty}{(q^{21}, q^{47}; q^{68})_\infty}
		= \sum_{n=0}^{\infty} \beta_n q^n,
		& \beta_{34n+7} = 0 \\ 
		\hline

		\frac{1}{{Y}_2^*(q)} = q^{-2} Y_2(q)
		= \frac{(q^{13}, q^{55}; q^{68})_\infty}{(q^{21}, q^{47}; q^{68})_\infty}
		= \sum_{n=0}^{\infty} \beta'_n q^n,
		& \beta'_{34n+11} = 0 \\ 
		\hline
		
		{Y_3^*(q)} = q^{-3} Y_3(q)
		= \frac{(q^{11}, q^{57}; q^{68})_\infty}{(q^{23}, q^{45}; q^{68})_\infty}
		= \sum_{n=0}^{\infty} \gamma_n q^n,
		& \gamma_{34n+30} = 0 \\ 
		\hline

		\frac{1}{{Y}_3^*(q)} = q^{-3} Y_3(q)
		= \frac{(q^{11}, q^{57}; q^{68})_\infty}{(q^{23}, q^{45}; q^{68})_\infty}
		= \sum_{n=0}^{\infty} \gamma'_n q^n,
		& \gamma'_{34n+2} = 0 \\ 
		\hline
		
		{Y_4^*(q)} = q^{-4} Y_4(q)
		= \frac{(q^{9}, q^{59}; q^{68})_\infty}{(q^{25}, q^{43}; q^{68})_\infty}
		= \sum_{n=0}^{\infty} \delta_n q^n,
		& \delta_{34n+15} = 0 \\ 
		\hline

		\frac{1}{{Y}_4^*(q)} = q^{-4} Y_4(q)
		= \frac{(q^{9}, q^{59}; q^{68})_\infty}{(q^{25}, q^{43}; q^{68})_\infty}
		= \sum_{n=0}^{\infty} \delta'_n q^n,
		& \delta'_{34n+23} = 0 \\ 
		\hline
		
		{Y_5^*(q)} = q^{-5} Y_5(q)
		= \frac{(q^{7}, q^{61}; q^{68})_\infty}{(q^{27}, q^{41}; q^{68})_\infty}
		= \sum_{n=0}^{\infty} \zeta_n q^n,
		& \zeta_{34n+30} = 0 \\ 
		\hline

		\frac{1}{{Y}_5^*(q)} = q^{-5} Y_5(q)
		= \frac{(q^{7}, q^{61}; q^{68})_\infty}{(q^{27}, q^{41}; q^{68})_\infty}
		= \sum_{n=0}^{\infty} \zeta'_n q^n,
		& \zeta'_{34n+6} = 0 \\ 
        \hline
		
		{Y_6^*(q)} = q^{-6} Y_6(q)
		= \frac{(q^{5}, q^{63}; q^{68})_\infty}{(q^{29}, q^{39}; q^{68})_\infty}
		= \sum_{n=0}^{\infty} \eta_n q^n,
		& \eta_{34n+7} = 0 \\ 
		\hline

		\frac{1}{{Y}_6^*(q)} = q^{-6} Y_6(q)
		= \frac{(q^{5}, q^{63}; q^{68})_\infty}{(q^{29}, q^{39}; q^{68})_\infty}
		= \sum_{n=0}^{\infty} \eta'_n q^n,
		& \eta'_{34n+19} = 0 \\ 
		\hline
		
		{Y_7^*(q)} = q^{-7} Y_7(q)
		= \frac{(q^{3}, q^{65}; q^{68})_\infty}{(q^{31}, q^{37}; q^{68})_\infty}
		= \sum_{n=0}^{\infty} \xi'_n q^n,
		& \xi'_{34n+14} = 0 \\ 
		\hline
		
		{Y_8^*(q)} = q^{-8} Y_8(q)
		= \frac{(q, q^{67}; q^{68})_\infty}{(q^{33}, q^{35}; q^{68})_\infty}
		= \sum_{n=0}^{\infty} \lambda_n q^n,
		& \lambda_{34n+17} = 0 \\ 
		\hline

		\frac{1}{{Y}_8^*(q)} = q^{-8} Y_8(q)
		= \frac{(q, q^{67}; q^{68})_\infty}{(q^{33}, q^{35}; q^{68})_\infty}
		= \sum_{n=0}^{\infty} \lambda'_n q^n,
		& \lambda'_{34n+33} = 0 \\ 
		\hline
		\hline
		
	\end{array}
	\]
    \section{Some partition-theoretic results}\label{s4}
	
	In this section, we show that color partition identities can be obtained from the theta-function identities
	established in Theorem (2.2) using color partitions of integers. As an example, we deduce two partition-theoretic
	identities from the theta-function identities of the continued fraction $Y_8(q)$. First, we will give the definition of color
	partitions of a positive integer $n$ and its generating function.
	
	A partition of a positive integer $n$ is a non-increasing sequence of positive integers, called parts, whose
	sum equals $n$. A part in a partition of $n$ is said to have $r$ colors if each part has $r$ copies and all of
	them are viewed as distinct objects. For any positive integer $n$ and $r$, let $C_r(n)$ denote the number of
	partitions of $n$ in which each part has $r$ distinct colors. 
	
	For example, if each part of a partition of $3$ has $2$ colors, say Pink (indicated by the suffix $p$) and
	Green (indicated by the suffix $g$), then the number of $2$-color partitions of $3$ is $10$, namely
	\[
	3_p,\, 3_g,\, 2_p+1_p,\, 2_p+1_g,\, 2_g+1_p,\, 2_g+1_g,\, 
	1_p+1_p+1_p,\, 1_p+1_p+1_g,\, 1_p+1_g+1_g,\, 1_g+1_g+1_g.
	\]
	
	The generating function of $C_r(n)$ is given by
	\begin{equation}
		\sum_{n=0}^{\infty} C_r(n) q^n = \frac{1}{(q;q)_\infty^{\,r}} .
		\tag{4.1}
	\end{equation}
	
	For positive integers $s,m$ and $r$, the quotient
	\begin{equation}
		\frac{1}{(q^s;q^m)_\infty^{\,r}}
		\tag{4.2}
	\end{equation}
	is the generating function of the number of partitions of $n$ with parts congruent to $s$ modulo $m$ and
	each part having $r$ colors. For example,
	\begin{equation}
		\frac{1}{(q^{s_1};q^m)_\infty^{\,r}(q^{s_2};q^m)_\infty^{\,r}}
		=
		\frac{1}{(q^{s_1},q^{s_2};q^m)_\infty^{\,r}}
		\tag{4.3}
	\end{equation}
	is the generating function of the number of partitions of a positive integer with parts congruent to $s_1$
	or $s_2$ modulo $m$, each part having $r$ distinct colors.	
	For convenience, we use the notation
	\begin{equation}
		(q^{r\pm};q^t)_\infty := (q^r,q^{t-r};q^t)_\infty,
		\tag{4.4}
	\end{equation}
	where $r$ and $t$ are positive integers and $r<t$.

    \begin{theorem}
        For any integer $n \geq 16$, let $D_1(n)$ denote the number of partitions of $n$ into parts $\equiv \pm 1, \pm 16, \pm 18$ or $\pm 34 \pmod{68}$ such that the parts $\equiv \pm 1$ and $\pm 34 \pmod{68}$ have $2$ colors. Let $D_2(n)$ denote the number of partitions of $n$ into parts 
        $\equiv \pm 16, \pm 18, \pm 33$ or $\pm 34 \pmod{68}$ such that parts $\equiv \pm 33$ and $\pm 34 \pmod{68}$ have $2$ colors. 
        Let $D_3(n)$ denote the number of partitions of $n$ into parts $\equiv \pm 1, \pm 17$ and $\pm 33 \pmod{68}$ with $2$ colors.
        Then
        \[D_1(n) - D_2(n-16) - D_3(n) = 0.\]
    \end{theorem}
\begin{proof}
Employing (1.3), (1.4), (1.5) and (1.25) in the theta-function identity (P) in Theorem 2.2, we obtain
\begin{equation}
\frac{(q^{33\pm}; q^{68})_{\infty}}{q^{8}(q^{1\pm}; q^{68})_{\infty}}
- q^{8} \frac{(q^{1\pm}; q^{68})_{\infty}}{(q^{33\pm}; q^{68})_{\infty}}
- \frac{(q^{16\pm}; q^{34})_{\infty}(q^{34\pm}; q^{34})_{\infty}^6}
{q^{8}(q^{1\pm}; q^{34})_{\infty}(q^{17}; q^{17})_{\infty}^2 (q^{68}; q^{68})_{\infty}^4}
= 0.
\tag{4.5}
\end{equation}

Changing the products in the third term of (4.5) to base $q^{68}$ and simplifying, we obtain
\begin{equation}
\frac{(q^{33\pm}; q^{68})_{\infty}}{(q^{1\pm}; q^{68})_{\infty}}
- q^{16} \frac{(q^{1\pm}; q^{68})_{\infty}}{(q^{33\pm}; q^{68})_{\infty}}
- \frac{(q^{16\pm}, q^{18\pm}; q^{68})_{\infty}(q^{34\pm}; q^{68})_{\infty}^2}
{(q^{1\pm}, q^{33\pm}; q^{68})_{\infty}(q^{17\pm}; q^{68})_{\infty}^2}
= 0.
\tag{4.6}
\end{equation}

Dividing (4.6) by 
$(q^{1\pm}, q^{16\pm}, q^{18\pm}, q^{33\pm}; q^{68})_{\infty}
(q^{34\pm}; q^{68})_{\infty}^2 ,$ we obtain
\begin{align}
	\frac{1}{(q^{16\pm}, q^{18\pm}; q^{68})_{\infty}(q^{1\pm}, q^{34\pm}; q^{68})_{\infty}^2}
	&- q^{16} \frac{1}{(q^{16\pm}, q^{18\pm}; q^{68})_{\infty}(q^{33\pm}, q^{34\pm}; q^{68})_{\infty}^2} \\
	&- \frac{1}{(q^{1\pm}, q^{17\pm}, q^{33\pm}; q^{68})_{\infty}^2} = 0
	\tag{4.7}
\end{align}

Equation (4.7) is equivalent to
\begin{equation}
\sum_{n=0}^{\infty} D_1(n) q^n
- q^{16} \sum_{n=0}^{\infty} D_2(n) q^n
- \sum_{n=0}^{\infty} D_3(n) q^n
= 0,
\tag{4.8}
\end{equation}
where we take $D_1(0) = D_2(0) = D_3(0) = 1$.

Comparing the coefficients of $q^n$ on both sides of (4.8), we complete the proof.
\end{proof}

\medskip

\textbf{Example.}
By enumerating the relevant partitions of $n = 16$, one can verify that 
\[
D_1(16) = 18, \quad D_2(0) = 1, \quad D_3(16) = 17,
\]
which satisfy Theorem (4.1)

\begin{theorem}
For any integer $n \geq 16$, let $K_1(n)$ denote the number of partitions of $n$ into parts 
$\equiv \pm 1, \pm 17, \pm 32$ or $\pm 34 \pmod{68}$ such that the parts $\equiv \pm 1$ and $\pm 17 \pmod{68}$ have $2$ colors. 

Let $K_2(n)$ denote the number of partitions of $n$ into parts 
$\equiv \pm 17, \pm 32, \pm 33$ or $\pm 34 \pmod{68}$ such that parts $\equiv \pm 17$ and $\pm 33 \pmod{68}$ have $2$ colors. 

Let $K_3(n)$ denote the number of partitions of $n$ into parts 
$\equiv \pm 1, \pm 16, \pm 18$ and $\pm 33 \pmod{68}$ such that the parts $\equiv \pm 1$ and $\pm 33 \pmod{68}$ have $2$ colors. 

Then
\[
K_1(n) + K_2(n-16) = K_3(n).
\]
\end{theorem}

\begin{proof}
    Proceeding as in the proof of Theorem 4.1, identity (o) in Theorem 2.2 can be expressed as
\begin{align}
	\frac{1}{(q^{32\pm}, q^{34\pm}; q^{68})_{\infty}(q^{1\pm}, q^{17\pm}; q^{68})_{\infty}^2}
	&+ q^{16} \frac{1}{(q^{32\pm}, q^{34\pm}; q^{68})_{\infty}(q^{17\pm}, q^{33\pm}; q^{68})_{\infty}^2} \\
	&= \frac{1}{(q^{16\pm}, q^{18\pm}; q^{68})_{\infty}(q^{1\pm}, q^{33\pm}; q^{68})_{\infty}^2}
	\tag{4.9}
\end{align}

Noting the generating functions, we write (4.9) as
\begin{equation}
\sum_{n=0}^{\infty} K_1(n) q^n
+ q^{16} \sum_{n=0}^{\infty} K_2(n) q^n
= \sum_{n=0}^{\infty} K_3(n) q^n,
\tag{4.10}
\end{equation}
where we take $K_1(0) = K_2(0) = K_3(0) = 1$.

The required result follows by comparing the coefficients of $q^n$ on both sides of (4.10).
\end{proof}

\medskip

\textbf{Example.}
By enumerating the relevant partitions of $n = 18$, one can verify that
\[
K_1(18) = 23, \quad K_2(2) = 0, \quad K_3(18) = 23,
\]
which satisfy Theorem (4.2).

	\bigskip
	\bigskip
	
	\noindent
	Department of Mathematics\\
	Ramanujan School of Mathematical Sciences\\
	Pondicherry University\\
	Puducherry- 605 014, India.\\

	\noindent Email: \texttt{iamdipikasarkar@pondiuni.ac.in}
	
	\noindent	Email: \texttt{dr.fathima.sn@pondiuni.ac.in} (\Letter)
	
\end{document}